\documentclass[11pt, a4paper]{article}
\usepackage{amsfonts,amsmath,amssymb,accents,amsthm}
\usepackage{verbatim}

\usepackage{hyperref, url}
\usepackage{graphicx}
\usepackage{color}
\usepackage[normalem]{ulem}

\newcommand{\noi}{\noindent}
\newcommand{\halmos}{\rule{1ex}{1.4ex}}
\newcommand{\QED}{\nopagebreak{\hspace*{\fill}$\halmos$\medskip}}

\newtheoremstyle{mythm}
  {}
  {}
  {\itshape}
  {}
  {\bfseries}
  {}
  {.5em}
  {#1 #2 \thmnote{(#3)}}

\theoremstyle{mythm}
\newtheorem{theorem}{Theorem}[section]
\newtheorem{proposition}[theorem]{Proposition}
\newtheorem{lemma}[theorem]{Lemma}
\newtheorem{exercise}[theorem]{Exercise}
\newtheorem{corollary}[theorem]{Corollary}
\newtheorem{conjecture}[theorem]{Conjecture}

\newtheorem{counterex}[theorem]{Counterexample}
\newtheorem{remark}[theorem]{Remark}

\newcommand{\bt}{\begin{theorem}}
\newcommand{\et}{\end{theorem}}
\newcommand{\bl}{\begin{lemma}}
\newcommand{\el}{\end{lemma}}
\newcommand{\bp}{\begin{proposition}}
\newcommand{\ep}{\end{proposition}}
\newcommand{\bcor}{\begin{corollary}}
\newcommand{\ecor}{\end{corollary}}
\newcommand{\br}{\begin{remark}\rm}
\newcommand{\er}{\end{remark}}
\newcommand{\bcon}{\begin{conjecture}}
\newcommand{\econ}{\end{conjecture}}
\newcommand{\bex}{\begin{exercise}}
\newcommand{\eex}{\end{exercise}}
\newcommand{\bcou}{\begin{counterex}}
\newcommand{\ecou}{\end{counterex}}

\newenvironment{Proof}[1][]{\noi\textbf{Proof #1}}{\QED}
\newcommand{\bpro}{\begin{Proof}}
\newcommand{\epro}{\end{Proof}}

\newcommand{\be}{\begin{equation}}
\newcommand{\ee}{\end{equation}}
\newcommand{\ba}{\begin{array}}
\newcommand{\ea}{\end{array}}
\newcommand{\bc}{\be\begin{array}{r@{\,}c@{\,}l}}
\newcommand{\ec}{\end{array}\ee}

\newcommand{\Ci}{{\cal C}}

\newcommand{\Ei}{{\cal E}}

\newcommand{\R}{{\mathbb R}}
\newcommand{\N}{{\mathbb N}}
\newcommand{\Z}{{\mathbb Z}}

\newcommand{\E}{{\mathbb E}}
\renewcommand{\P}{{\mathbb P}}

\newcommand{\Aston}[1]{\underset{{#1}\downarrow 0}{\Longrightarrow}}

\newcommand{\p}{\mathrm P}
\newcommand{\e}{\mathrm E}

\makeatletter
\let\@fnsymbol\@arabic
\makeatother

\begin{document}

\makeatletter\@addtoreset{equation}{section}
\makeatother\def\theequation{\thesection.\arabic{equation}}

\renewcommand{\labelenumi}{{\rm (\roman{enumi})}}
\renewcommand{\theenumi}{\roman{enumi}}

\title{Scaling Limits of Disorder Relevant Non-Binary Spin Systems}
\author{ Liuyan Li \footnote{School of Statistics and Data Science,
            Guangdong University of Finance and Economics,
            Guangzhou, 510320, P. R. China.
            E-mail: mathlily@gdufe.edu.cn}
	\and Vlad Margarint \footnote{Office 350E, Fretwell Building,
	        University of North Carolina Charlotte, USA.
	        Email: vmargari@charlotte.edu}
	\and Rongfeng Sun \footnote{Department of Mathematics,
		National University of Singapore,
		10 Lower Kent Ridge Road, 119076 Singapore.
		Email: matsr@nus.edu.sg}
}

\date{\today}

\maketitle

\begin{abstract}\noi
In \cite{CSZ17a}, Caravenna, Sun and Zygouras gave general criteria for the partition functions of binary valued spin systems with a relevant random field perturbation to have non-trivial continuum and weak disorder limits. In this work, we show how these criteria can be extended to non-binary valued spin systems.
\end{abstract}
\vspace{.5cm}

\noi
{\it MSC 2010.} Primary: 82B44, Secondary: 82D60, 60K35.\\
{\it Keywords.} Continuum Limit, Disorder Relevance, Polynomial Chaos, Wiener Chaos. \\


\section{Introduction}

We consider here equilibrium statistical mechanics models defined on a lattice, which interact
with a random environment (disorder) in the form of a random external field. If we consider
the random field as a perturbation of the underlying model without disorder, then such disorder
perturbation is called a {\em relevant perturbation} if the presence of disorder, regardless of
its strength, changes the large scale qualitative behaviour of the model (i.e., changes the
critical exponents of the model and leads to a different scaling limit than the model without
disorder). The disorder perturbation is called {\em irrelevant} if a small amount of disorder does
not change the large scale behaviour (i.e., the scaling limit is the same as the model without
disorder). Disorder relevance vs irrelevance often depends on the dimension of
the underlying model. At the critical dimension, whether disorder is relevant or irrelevant is
much more subtle and is often referred to as {\em marginal relevance} or {\em irrelevance}. We
refer the reader to \cite{Gi11} for more background on disordered systems and  the {\it Harris criterion}
\cite{H74} on when disorder perturbation is predicted by physicists to be relevant/irrelevant.

A classic example that fits into this framework is the Directed Polymer Model (DPM),
where a directed polymer is modelled by a random walk $X$ on $\Z^d$ interacting with an i.i.d.\ space-time random
environment (disorder) $\omega:=(\omega(n,x))_{n\in \N, x\in \Z^d}$. Given $\omega$, polymer length $N$,
and inverse temperature $\beta\geq 0$, the polymer measure is defined by weighting
each random walk path $(X_n)_{0\leq n\leq N}$ with a Gibbs weight $e^{\beta \sum_{n=1}^N \omega(n, X_n)}/Z^\omega_{N, \beta}$,
where the normalizing constant $Z^\omega_{N, \beta}$ is known as the partition function.
We can thus regard the DPM as a disorder perturbation of the random walk. The configuration of the
random walk $(X_n)_{0\leq n\leq N}$ can be identified with a binary-valued spin field $\sigma(n, x)\in \{0, 1\}$
with $\sigma(n, x)= 1_{\{X_n=x\}}$, and the disorder $\omega(n, x)$ can be interpreted as a random external
field acting on the spin $\sigma(n, x)$. It is known that in all dimensions $d\geq 1$ (see the book \cite{C17} and the
references therein), there is a critical point $\beta_c(d)\geq 0$ such that, for $\beta<\beta_c(d)$, the polymer has
diffusive fluctuations and converges to the same Brownian motion as the underlying random walk $X$. On the other hand
when $\beta>\beta_c(d)$, the polymer is expected to be super-diffusive, i.e., $\Vert X_N\Vert \gg N^{1/2}$ with high
probability under the polymer measure. In $d \geq 3$, it is known that $\beta_c(d)>0$, which means that
small disorder does not alter the large scale behaviour of the random walk and hence disorder perturbation is irrelevant.
On the other hand, it is known that $\beta_c(d)=0$ for $d=1, 2$, which means that disorder perturbation is relevant.
Dimension $d=2$ turns out to be the critical dimension between disorder relevance and irrelevance, and hence disorder
is marginally relevant in $d=2$.

From the renormalization group point of view, disorder relevance means that, the effective disorder strength of
the rescaled model diverges as we zoom out on larger and larger scales (or equivalently, take the continuum limit by
sending the lattice spacing to $0$), while disorder irrelevance means that the effective disorder strength vanishes
in the large scale limit. This suggests that, for disorder relevant models, it may be possible to tune the strength
of disorder down to $0$ at a suitable speed as we send the lattice spacing to $0$, such that we obtain in the limit a continuum
model with non-trivial dependence on disorder. Furthermore, for disordered systems defined via Gibbs measures, such
non-trivial continuum limits should already appear at the level of partition functions. This was the key insight
of \cite{CSZ17a}, which took inspiration from an earlier result of this type for the DPM in $d=1$ \cite{AKQ14a, AKQ14b}
and developed convergence criteria for partition functions of general (non-marginal) disorder relevant systems. These
convergence criteria were then applied in \cite{CSZ17a} to the disordered pinning model, the
long-range DPM in dimension $1+1$, and the random field perturbation of the critical 2-dimensional Ising model. Subsequently,
the continuum disordered pinning model and the continuum 2D random field Ising model (RFIM) were constructed in \cite{CSZ16}
and \cite{BS22} respectively.

Roughly speaking, the convergence criteria formulated in \cite{CSZ17a} require: 1) the correlation
functions of the underlying spin system (without disorder) have a non-trivial continuum limit (see {\bf (A1)} below), which
is expected to hold for a spin system at the critical point of a continuous phase transition; 2) the correlation functions
are square integrable (see {\bf (A1)} below) and have sufficiently fast decay in the order of the correlation functions (see {\bf (A2)}
below), which ensures (non-marginal) disorder relevance and the continuum limit of the partition functions admit a Wiener-It\^o chaos
expansion (see \eqref{eq:wtZ3} below). See \cite{CSZ17a} and the short review \cite{CSZ16b} for more details and further discussions.

One limitation of \cite{CSZ17a} is that, the convergence criteria were formulated for the partition functions of (non-marginal)
disorder relevant spin systems with {\em binary valued} spins, which played an essential role in the proof (see \eqref{eq:Mayer} below). Our purpose here
is to extend \cite{CSZ17a} and formulate convergence criteria for random field perturbations of {\em non-binary valued} spin systems,
such that their partition functions admit non-trivial continuum limits. Such convergence criteria could potentially be applied
to critical 2D Ising models with more general single spin measures and used to establish the universality of the continuum RFIM
constructed in \cite{BS22}. It could also be used to investigate non-trivial continuum limits of other disorder relevant systems,
such as the Blume-Capel model, where the spins can take values in the set $\{-1, 0, 1\}$ and the model has a more complex phase
diagram (see \cite{BP18, GKP24} and the references therein). However, verifying the convergence criteria {\bf (A1)-(A3)} below for concrete
models will be challenging, because it requires sharp control on the underlying spin system (without disorder) at its critical point.
Such control is available so far only for models which have a time direction (disordered pinning and DPM), or for the critical 2D Ising model
which is integrable to a large degree.

As in \cite{CSZ17a}, our convergence criteria do not apply to systems for which disorder is marginally relevant, such as the DPM
in the critical dimension $2$. Neither do we expect general convergence criteria to exist for marginally relevant models. However,
it is worth noting that there has been significant progress in understanding the scaling limit of the 2D DPM, which is closely connected to
the 2D stochastic heat equation (SHE), a critical singular stochastic partial differential equation (SPDE). In fact, there is a parallel between disordered systems and the theory of singular SPDEs \cite{H14, GIP15}, where the driving noise plays the role of disorder perturbation of
the underlying PDE, and the notions of sub-criticality, criticality, and super-criticality for singular SPDEs correspond respectively to the
notions of disorder relevance, marginal relevance/irrelevance, and disorder irrelevance for disordered systems. For more details on
recent progress on the 2D DPM and SHE, see the recent lecture notes \cite{CSZ24} and the references therein. It will be extremely interesting
to find a marginally relevant spin system (without a time direction) that have a non-trivial disordered continuum limit.

\subsection{Setup}
We will follow the same setup as in \cite{CSZ17a}. First we introduce the underlying spin system.
For $d \geq 1$, we consider a bounded, open, and simply connected domain $\Omega \subseteq \mathbb{R}^{d}$, and we define its
lattice approximation by $\Omega_{\delta}:=\Omega \cap (\delta \mathbb{Z})^{d}$ for $\delta>0$. A spin $\sigma_x\in \R$
is assigned to each $x\in \Omega_\delta$, and let $\p_{\Omega_{\delta}}^{\text {ref }}$ (with expectation
$\e_{\Omega_{\delta}}^{\text {ref }}$) be a probability measure on the spin configuration
$\sigma=\left(\sigma_{x}\right)_{x \in \Omega_{\delta}}$. Typically, $\p_{\Omega_{\delta}}^{\text {ref }}$
will be chosen to be the law of an equilibrium spin system at the critical point of a continuous phase transition, such
that the spin field $\sigma$ admits a non-trivial continuum limit as $\delta\downarrow 0$. Sometimes we will drop the sub
and superscripts in $\p_{\Omega_{\delta}}^{\text {ref }}$ to simplify the notation. For $x\in \Omega$, we will denote by
$x^\delta$ the point in $\Omega_{\delta}=\Omega \cap (\delta \mathbb{Z})^{d}$ that is closest to $x$ (fix any convention to break the tie if such $x^\delta$ is not unique).

Next, we introduce the random field (disorder), which is given by a family of i.i.d.\ random variables
$\omega:=\left(\omega_{x}\right)_{x \in \Omega_{\delta}}$ with $\E[\omega_x]=0$ and $\E[\omega_x^2]=1$. Probability and expectation for
$\omega$ will be denoted by $\P$ and $\E$. We assume that $\omega$ has finite log moment generating function
\begin{align}\label{2.1}
 \phi(\lambda) := \log \E[e^{\lambda \omega_x}] \qquad \mbox{ for all } |\lambda|<\lambda_0 \ \ \mbox{for some  } \lambda_0>0.
\end{align}
Note that under suitable scaling, the i.i.d.\ field $(\omega_x)_{x\in (\delta \Z)^d}$ converges in the continuum limit to a white
noise $W$ on $\R^d$, which is a Gaussian process $W=(W(f))_{f \in L^{2}(\mathbb{R}^{d})}$ with $\mathbb{E}[W(f)]=0$ and $\text{Cov}(W(f), W(g))=\int_{\mathbb{R}^{d}}f(x)g(x){\rm d} x$.

Given the random field (disorder) $\omega$ and disorder strength $\lambda>0$, we can then define the random field perturbation of
$\p_{\Omega_{\delta}}^{\text {ref }}$ through the following $\omega$-dependent Gibbs measure
\begin{equation}\label{eq:Gibbs}
\p_{\Omega_{\delta} ; \lambda}^{\omega}(\mathrm{d} \sigma):=\frac{e^{\sum_{x \in \Omega_{\delta}} \lambda \omega_{x} \sigma_{x}}}{Z_{\Omega_{\delta} ; \lambda}^{\omega}} \mathrm{P}_{\Omega_{\delta}}^{\mathrm{ref}}(\mathrm{d} \sigma),
\end{equation}
where the normalizing constant, called the partition function, is defined by
\begin{equation}\label{eq:Zom}
Z_{\Omega_{\delta} ; \lambda}^{\omega}:=\mathrm{E}_{\Omega_{\delta}}^{\mathrm{ref}}\left[e^{\sum_{x \in \Omega_{\delta}} \lambda \omega_{x}\sigma_{x}}\right] .
\end{equation}
The question we want to address is whether there is a suitable choice of $\lambda=\lambda_\delta \downarrow 0$, sometimes called {\em intermediate disorder scaling}, such that under suitable centering and scaling, $Z_{\Omega_{\delta} ; \lambda_\delta}^{\omega}$ admits non-trivial distributional limits. This is usually the first step in showing that the random Gibbs measure $\p_{\Omega_{\delta} ; \lambda}^{\omega}$ also has a non-trivial disordered continuum limit, where the disorder in the continuum limit is given by the white noise $W$. For binary-valued spin systems, this question was addressed in \cite{CSZ17a} where general convergence criteria were formulated. Our goal here is to consider the case where the spins $\sigma_x$ are not binary-valued, although we still assume boundedness to avoid additional technical complications.

\textbf{(A0)} There exists $K>0$ such that $\p_{\Omega_{\delta}}^{\mathrm{ref}}(\sigma_x\in [-K, K])=1$ for all $x\in \Omega_\delta$
and $\delta>0$.

To motivate our main result, we first recall how the convergence of $Z_{\Omega_{\delta} ; \lambda_\delta}^{\omega}$ in law was
established in \cite[Section 8]{CSZ17a} for spin systems with $\sigma_x\in \{\pm 1\}$. The starting point is the expansion
\begin{align}
Z_{\Omega_{\delta} ; \lambda_\delta}^{\omega} &=\mathrm{E}_{\Omega_{\delta}}^{\mathrm{ref}}\Big[e^{\sum_{x \in \Omega_{\delta}}\lambda_\delta \omega_{x}\sigma_{x}}\Big] =\mathrm{E}_{\Omega_{\delta}}^{\mathrm{ref}}\Big[\prod_{x \in \Omega_\delta}e^{\lambda_\delta \omega_{x}\sigma_{x}}\Big] \notag \\
& = \mathrm{E}_{\Omega_{\delta}}^{\mathrm{ref}}\Big[\prod_{x \in \Omega_\delta} (\cosh \lambda_\delta \omega_x + \sigma_x \sinh \lambda_\delta\omega_x )\Big] \label{eq:Mayer} \\
&= \exp\Big\{\sum_{x \in \Omega_\delta} \log \cosh \lambda_\delta \omega_x\Big\} \mathrm{E}_{\Omega_{\delta}}^{\mathrm{ref}}\Big[\prod_{x \in \Omega_\delta} (1+ \sigma_x \tanh \lambda_\delta\omega_x )\Big], \notag
\end{align}
where we linearised $e^{\lambda_\delta \omega_{x}\sigma_{x}}$ in \eqref{eq:Mayer} since $\sigma_x$ is binary valued. Note that
$\sum_{x \in \Omega_\delta} \log \cosh \lambda_\delta \omega_x$ is a sum of i.i.d.\ random variables with
\begin{align*}
\E\Big[\sum_{x \in \Omega_\delta} \log \cosh \lambda_\delta \omega_x\Big]  = |\Omega_\delta| \Big(\frac{\lambda_\delta^2}{2} +O(\lambda_\delta^4)\Big) \quad \mbox{and} \quad
\mathbb{V}{\rm ar}\Big(\sum_{x \in \Omega_\delta} \log \cosh \lambda_\delta \omega_x\Big)  \leq C |\Omega_\delta| \lambda_\delta^4.
\end{align*}
As long as $|\Omega_\delta| \lambda_\delta^4 = O(\delta^{-d} \lambda_\delta^4) \to 0$ as $\delta\downarrow 0$, i.e., $\lambda_\delta = o(\delta^{d/4})$, then
\begin{equation}\label{eq:LLN}
e^{-\frac{\lambda_\delta^2}{2} |\Omega_\delta|} \exp\Big\{\sum_{x \in \Omega_\delta} \log \cosh \lambda_\delta \omega_x\Big\} \to 1 \quad \mbox{ in probability},
\end{equation}
and hence we can normalize $Z_{\Omega_{\delta} ; \lambda_\delta}^{\omega}$ and then focus on the distributional limit of
\begin{equation}\label{eq:wtZ}
\widetilde Z_{\Omega_{\delta} ; \lambda_\delta}^{\omega} := \mathrm{E}_{\Omega_{\delta}}^{\mathrm{ref}}\Big[\prod_{x \in \Omega_\delta} (1+ \sigma_x \tanh \lambda_\delta\omega_x )\Big] = 1 + \sum_{k=1}^\infty \sum_{I\subset \Omega_\delta, |I|=k} \mathrm{E}_{\Omega_{\delta}}^{\mathrm{ref}}[\sigma_I] \prod_{x\in I} \xi_\delta(x),
\end{equation}
where $\sigma_I:= \prod_{x\in I} \sigma_x$, and $\xi_\delta(x):= \tanh \lambda_\delta\omega_x$ is a family of i.i.d.\ random variables with
\begin{equation}
\E[\xi_\delta(x)] = O(\lambda_\delta^3) \quad \mbox{and} \quad \mathbb{V}{\rm ar}(\xi_\delta(x))= \lambda_\delta^2 +O(\lambda_\delta^4) \qquad
\mbox{by Taylor expansion}.
\end{equation}
The expansion in \eqref{eq:wtZ} is called a polynomial chaos expansion in the family of random variables $\xi_\delta$. Pretending that the mean of $\xi_\delta(x)$ is zero (because it is negiligible as $\lambda_\delta\to 0$), it is an $L^2$-orthogonal expansion and can be regarded as the discrete analogue of the Wiener-It\^{o} chaos expansion w.r.t.\ a white noise $W$. In particular, if we match mean and variance (up to leading order)
and make the approximation
$$
\xi_\delta(x) \approx \lambda_\delta \delta^{-d/2} \int_{\Lambda_\delta(x)} W({\rm d}y),
$$
where $\Lambda_\delta(x)$ denotes the cube of side length $\delta$ centered at $x\in \Omega_\delta$. Then we have
\begin{equation}\label{eq:wtZ2}
\widetilde Z_{\Omega_{\delta} ; \lambda_\delta}^{\omega} \approx  1 + \sum_{k=1}^\infty \frac{1}{k!} \idotsint_{\Omega^k}
\lambda_\delta^k \delta^{-dk/2} \mathrm{E}_{\Omega_{\delta}}^{\mathrm{ref}}[\sigma_{x^\delta_1}\cdots \sigma_{x^\delta_k}]\, W({\rm d}x_1)\cdots W({\rm d}x_k),
\end{equation}
where $x^\delta_i$ the point in $\Omega_{\delta}$ closest to $x_i\in \Omega$, and the factor $1/k!$ arises because $\Omega^k$ consists
of $k!$ sectors that give identical contributions by exchangeability. Suppose the spin field $\sigma$ with law $\mathrm{P}_{\Omega_{\delta}}^{\mathrm{ref}}$ satisfies the assumption

\textbf{(A1)} There exists $\gamma>0$ such that for any $k\in\N$, the rescaled $k$-spin correlation function
\begin{equation}\label{eq:psidelta}
\psi_{\delta}(x_1, \ldots, x_k) := \bold{1}_{\{x_i^{\delta} \neq x_j^{\delta} \, \forall\, i \neq j \}}\, \delta^{-k\gamma}\, \mathrm{E}_{\Omega_{\delta}}^{\mathrm{ref}}[\sigma_{x_1^{\delta}} \ldots \sigma_{x_k^{\delta}}], \quad (x_1, \ldots, x_k)\in \Omega^k,
\end{equation}
converges in $L^2(\Omega^k)$ to some limit
$\psi_0 :\Omega^k \to \R$. More precisely,
\begin{equation}\label{A1}
\lim _{\delta \downarrow 0}\left\|{\psi}_{\delta}-\psi_{0}\right\|_{L^{2}(\Omega^k)}^{2}=0.
\end{equation}
We will regard $\psi_0$ as a function defined on $\cup_{k\in\N}\Omega^k$ so that it can take arbitrary number of arguments.
Note that assumption {\bf (A1)} ensures that $\p_{\Omega_{\delta}}^{\text {ref }}$ is the law of a spin system at the {\em critical point
of a continuous phase transition}, which admits a non-trivial continuum limit as $\delta\downarrow 0$. Furthermore, the finite $L^2$ norm
assumption on $\psi_0$ essentially guarantees that the system is disorder relevant (see \cite[Section 1.3]{CSZ17a}) and the iterated stochastic
integrals in \eqref{eq:wtZ3} below are well-defined.

If {\bf (A1)} is satisfied, then by \cite[Theorem 2.3]{CSZ17a}, we can choose $\lambda_\delta:= \hat\lambda \delta^{\frac{d}{2}-\gamma}$ in \eqref{eq:wtZ2} to obtain that, as $\delta\downarrow 0$,
\begin{equation}\label{eq:wtZ3}
\widetilde Z_{\Omega_{\delta} ; \lambda_\delta}^{\omega} \stackrel{\rm dist}{\longrightarrow} 1 + \sum_{k=1}^\infty \frac{\hat \lambda^k}{k!} \idotsint_{\Omega^k}
 \psi_0(x_1, \ldots, x_k)\, {\rm d} W(x_1)\cdots {\rm d}W(x_k),
\end{equation}
provided this series is convergent in $L^2$, and there is uniformity (in $\delta$) in the convergence of the series $\sum_k$ in \eqref{eq:wtZ2}
so that the series can be truncated at a large but fixed value of $k=M$ as $\delta\downarrow 0$. This leads to the second assumption

\textbf{(A2)} For any $\hat\lambda>0$,
\begin{equation}\label{eq:A2}
\lim_{M\to\infty} \limsup_{\delta\downarrow 0} \sum_{k=M+1}^\infty \frac{\hat\lambda^{2k}}{k!} \Vert \psi_\delta\Vert_{L^{2}(\Omega^k)}^{2} = 0.
\end{equation}

\begin{remark}{\rm
With $\lambda_\delta:= \hat\lambda \delta^{\frac{d}{2}-\gamma}$, the condition $\lambda_\delta = o(\delta^{d/4})$ that ensures \eqref{eq:LLN}
becomes $\gamma <d/4$. When $\p_{\Omega_{\delta}}^{\text {ref }}$ is the law of the critical $2$-dim Ising model, assumptions {\bf (A1)}
and {\bf (A2)} were verified in \cite{CSZ17a} with $\gamma =1/8<d/4=1/2$. }
\end{remark}

\subsection{Main result}
When the spins $\sigma_x$ are not binary valued, the linearization step \eqref{eq:Mayer} can no longer be applied. Furthermore, it is not clear
whether it is even possible to normalise the partition function as in \eqref{eq:LLN} that could lead to a polynomial chaos expansion as in
\eqref{eq:wtZ}. Therefore we will consider instead the modified partition function
\begin{equation}\label{eq:ZWick}
\widehat Z_{\Omega_{\delta} ; \lambda}^{\omega}:=\e_{\Omega_{\delta}}^{\mathrm{ref}}\left[e^{\sum_{x \in \Omega_{\delta}} \big(\lambda \omega_{x}\sigma_{x}-\phi(\lambda \sigma_x)\big)}\right],
\end{equation}
which automatically satisfies $\E[\widehat Z_{\Omega_{\delta} ; \lambda}^{\omega}]=1$. When $\omega_x$ are i.i.d.\ standard Gaussian,
such a normalization corresponds to replacing $e^{\lambda \omega_{x}\sigma_{x}}$ by the Wick exponential
$:e^{\lambda \omega_{x}\sigma_{x}}:\, = e^{\lambda \omega_{x}\sigma_{x}- \lambda^2\sigma_x^2/2}$.
Although this changes the underlying Gibbs measure $\p^\omega_{\Omega_\delta, \lambda}$ defined in \eqref{eq:Gibbs} because $\sum_x \phi(\lambda \sigma_x)$ depends on the spin configuration $\sigma$, there is intrinsic interest in such a Gibbs measure as seen in the case of the
two-dimensional Parabolic Anderson Model (PAM) \cite{QRV22}. We will discuss this in more detail in Remark \ref{R:PAM}.

To show that $\widehat Z_{\Omega_{\delta} ; \lambda}^{\omega}$ defined in \eqref{eq:ZWick} has a non-trivial continuum limit, our basic strategy
is to perform Taylor expansion. In contrast to \eqref{eq:wtZ}, we will no longer have an $L^2$-orthogonal expansion due to the presence of higher
powers of $\lambda \omega_x\sigma_x$. But the leading order terms are expected to be comparable to the expansion in \eqref{eq:wtZ}. Therefore most of the work goes into controlling the higher order terms in the expansion and show they are negligible as $\delta\downarrow 0$. This requires
one more assumption to control the $k$-spin correlations when some of the $k$ spins coincide:

\textbf{(A3)} For $k\geq 1$, let $x^\delta_1, \ldots, x^\delta_k$ be distinct points in $\Omega_\delta$. Let $r_1, \ldots, r_k\in\N$ and denote $(r_i)_2 := r_i \mbox{ (mod 2)} \in \{0, 1\}$. Then there exists a universal constant $C\geq 1$ such that
\begin{align}\label{A3}
\Big| \e_{\Omega_{\delta}}^{\mathrm{ref}}\Big[\prod_{i=1}^k\sigma_{x^\delta_i}^{r_i}\Big]\Big| \leq C^{\sum_{i=1}^k (r_i-(r_i)_2)} \Big|\e_{\Omega_{\delta}}^{\mathrm{ref}}\Big[\prod_{i=1}^k \sigma_{x^\delta_i}^{(r_i)_2}\Big] \Big|.
\end{align}
In light of the scaling property in assumption {\bf (A1)}, we believe assumption \textbf{(A3)} can be weakened further by allowing a diverging
constant $C=C(\delta)$, possibly as a negative power of $\delta$. But we will not pursue it here.

\begin{remark}{\rm
For the Ising model with $\sigma_x\in \{\pm 1\}$, \eqref{A3} holds trivially with $C=1$. In general, when spins fuse, we expect there to be a
pairing effect that allows one to replace $\sigma_{z}^{r}$ by $\sigma_{z}^{(r)_2}$ for $r\in \N$. For a centred Gaussian field $(\sigma_{x^\delta})_{x^\delta\in \Omega_\delta}$ with a covariance $\mathbb{C}{\rm ov}(\sigma_{x^\delta}, \sigma_{y^\delta})$ that decays polynomially in $|x^\delta -y^\delta|/\delta$, this can be verified using Wick's theorem. A similar result has been established for the critical site percolation on the planar triangular lattice, see \cite[Theorem 1.1]{CF24}, even though Wick's theorem does not apply in this setting.}
\end{remark}

We are now ready to state our main result.
\begin{theorem}\label{T:main}
Assume that the reference spin measure $\p^{\mathrm{ref}}_{\Omega_\delta}$ satisfies assumptions \textbf{(A0)}-\textbf{(A3)} for some $K, \gamma>0$ and $\psi_0$, with $\gamma<d/4$. Then with $\lambda=\lambda_\delta:=\hat \lambda \delta^{\frac{d}{2}-\gamma}$, the normalized partition function $\widehat Z_{\Omega_{\delta} ; \lambda_\delta}^{\omega}$ defined in \eqref{eq:ZWick} converges in distribution (as $\delta\downarrow 0$) to a non-trivial limit
\begin{equation}\label{eq:MCalZ}
{\mathcal Z}^W_{\Omega, \hat\lambda} = 1 + \sum_{k=1}^\infty \frac{\hat\lambda^k}{k!} \idotsint_{\Omega^k} \psi_0(x_1, \ldots, x_k) \, {\rm d}W(x_1) \ldots {\rm d}W(x_k),
\end{equation}
where $W$ is a white noise on $\R^d$ and the series converges in $L^2$.
\end{theorem}

\begin{remark}[\textbf{PAM}] \label{R:PAM}{\rm
The Wick ordering of the exponential in \eqref{eq:ZWick} has appeared in the study of SPDEs such as
the Parabolic Anderson Model (PAM)
\begin{equation}\label{eq:PAM1}
\frac{\partial}{\partial t}u(t, x) = \frac{1}{2}\Delta u(t, x) +  (u\cdot W)(t,x), \qquad t\geq 0, x\in \R^d,
\end{equation}
where $W$ is a white noise on $\R^d$. The equation is singular in $d\geq 2$ due to the product $u\cdot W$.
Before the theory of regularity structures \cite{H14} (see also later works \cite{HL15, HL18} on the
PAM), this difficulty is often bypassed by replacing the singular product $u\cdot W$ by the
Wick product $u \diamond W$ (see \cite{NZ89, Hu02, HY09, QRV22}),
\begin{equation}\label{eq:PAM2}
\frac{\partial}{\partial t}\hat u(t, x) = \frac{1}{2}\Delta \hat u(t, x) +  (\hat u \diamond W)(t,x), \qquad t\geq 0, x\in \R^d.
\end{equation}
By discretising time and space, we obtain approximations of $u$ and $\hat u$ which admit the Feynman-Kac representations
\begin{align*}
u^\delta(t, x) & := \e\Big[ e^{\sum_{z \in \delta \Z^d} \lambda_\delta \omega_z L(t^\delta, z) } \Big], \\
\hat u^\delta(t, x) & := \e\Big[ e^{\sum_{z \in \delta \Z^d} \left(\lambda_\delta \omega_z L(t^\delta, z) - \phi(\lambda_\delta L(t^\delta, z))\right)} \Big],
\end{align*}
where $\e[\cdot]$ denotes expectation for a random walk on $\delta \Z^d$ with local time $L(t^\delta, \cdot)$ at time $t^\delta :=t\delta^{-2}$, and $(\omega_z)_{z\in \delta \Z^d}$ are i.i.d.\ standard normals that discretise the white noise $W$. We note that $u^\delta$ and $\hat u^\delta$ are instances of $Z_{\Omega_{\delta} ; \lambda}^{\omega}$ in \eqref{eq:Zom} and $\widehat Z_{\Omega_{\delta} ; \lambda}^{\omega}$ in \eqref{eq:ZWick}, with spin values $\sigma_z = L(t^\delta, z)$ for $z\in \delta \Z^d$. Although $\hat u^\delta$ and its associated Gibbs measure (polymer measure) $\widehat \p^\omega_{\delta; \lambda_\delta}$ on the underlying random walk is less physical than $u^\delta$ and its associated polymer measure $\p^\omega_{\delta; \lambda_\delta}$, it was recently pointed out in \cite[Example 59]{QRV22} that in the planar case $d=2$, one could recover the continuum limit of $\p^\omega_{\delta; \lambda_\delta}$ from that of $\widehat\p^\omega_{\delta; \lambda_\delta}$ via a change of measure because under the measure $\widehat\p^\omega_{\delta; \lambda_\delta}$,  the weight factor $e^{\sum_{z \in \delta \Z^d} \phi(\lambda_\delta L(t^\delta, z))}$ has a well-defined limit. This suggests that Theorem \ref{T:main} could also be the first step towards identifying the
continuum limit of $Z_{\Omega_{\delta} ; \lambda}^{\omega}$ and its associated Gibbs measure $\p_{\Omega_{\delta} ; \lambda}^{\omega}$.
}
\end{remark}

\begin{remark}[Unbounded Spins]{\rm
We can formulate alternative convergence criteria in Theorem \ref{T:main} without the assumption {\bf (A0)} that the spins $\sigma_x$ are uniformly
bounded. The assumption $|\sigma_x|\leq K$ is only used in the proof of Theorem \ref{T:main} to ensure that $\lambda_\delta K$ is sufficiently small
as $\delta\downarrow 0$. Since $\lambda_\delta= \hat \lambda \delta^{\frac{d}{2}-\gamma}$, it is enough that $|\sigma_x|\leq K_\delta$ with $K_\delta\ll \delta^{\gamma-\frac{d}{2}}$. Therefore we can approximate $\widehat Z_{\Omega_{\delta} ; \lambda}^{\omega}$ in \eqref{eq:ZWick} by
\begin{equation}\label{eq:ZWick2}
\widetilde Z_{\Omega_{\delta} ; \lambda}^{\omega}:=\e_{\Omega_{\delta}}^{\mathrm{ref}}\Bigg[e^{\sum_{x \in \Omega_{\delta}} \big(\lambda \omega_{x}\sigma_{x}-\phi(\lambda \sigma_x)\big)} \prod_{x\in \Omega_\delta} 1_{\{|\sigma_x|\leq K_\delta\}}\Bigg].
\end{equation}
The $L^1$-norm of the error of this approximation can be bounded by
\begin{equation}\label{eq:A0'}
\E[\widehat Z_{\Omega_{\delta} ; \lambda}^{\omega}-\widetilde Z_{\Omega_{\delta} ; \lambda}^{\omega}] = \p_{\Omega_{\delta}}^{\mathrm{ref}}( |\sigma_x|>K_\delta \mbox{ for some } x\in \Omega_\delta).
\end{equation}
Therefore the assumption {\bf (A0)} can be replaced by the assumption {\bf (A0')}: For some choice of $K_\delta\ll \delta^{\gamma-\frac{d}{2}}$,
the r.h.s.\ of \eqref{eq:A0'} tends to $0$ as $\delta\downarrow 0$. For the same choice of $K_\delta$, we can replace $\p_{\Omega_{\delta}}^{\mathrm{ref}}$ by the conditional law
$$
\widetilde \p_{\Omega_{\delta}}^{\mathrm{ref}}(\cdot) = \p_{\Omega_{\delta}}^{\mathrm{ref}}(\,\cdot\, | \, \forall\, x\in \Omega_\delta, |\sigma_x|\leq K_\delta)
$$
and rewrite
\begin{equation*}
\widetilde Z_{\Omega_{\delta} ; \lambda}^{\omega} =\p_{\Omega_{\delta}}^{\mathrm{ref}}\big(|\sigma_x|\leq K_\delta \, \forall\, x\in \Omega_\delta\big)
\cdot \widetilde \e_{\Omega_{\delta}}^{\mathrm{ref}}\Big[e^{\sum_{x \in \Omega_{\delta}} \big(\lambda \omega_{x}\sigma_{x}-\phi(\lambda \sigma_x)\big)}\Big].
\end{equation*}
Assumptions {\bf (A1)-(A3)} should then also be modified accordingly with $\e_{\Omega_{\delta}}^{\mathrm{ref}}[\cdot]$ therein replaced by $\widetilde\e_{\Omega_{\delta}}^{\mathrm{ref}}[\cdot]$.
}
\end{remark}

\section{Proof of Theorem \ref{T:main}}

In this section, we prove the main result of the paper. Recall from \eqref{eq:ZWick} the normalised partition function
$$
\widehat Z_{\Omega_{\delta} ; \lambda}^{\omega}:=\e_{\Omega_{\delta}}^{\mathrm{ref}}\left[e^{\sum_{x \in \Omega_{\delta}} \big(\lambda \omega_{x}\sigma_{x}-\phi(\lambda \sigma_x)\big)}\right], \qquad \mbox{where} \quad \lambda=\lambda_\delta=\hat \lambda \delta^{\frac{d}{2}-\gamma}.
$$
For $x \in \Omega_\delta$, let $\eta_x=\eta_x(\sigma, \omega):= e^{\lambda\omega_x\sigma_x-\phi(\lambda\sigma_x)}-1-\lambda\omega_x\sigma_x.$
The starting point of our analysis is the expansion
\begin{align}
\widehat Z_{\Omega_{\delta} ; \lambda}^{\omega}&=\mathrm{E}_{\Omega_{\delta}}^{\mathrm{ref}}\Big[\prod_{x \in \Omega_\delta}(1+\lambda\omega_x\sigma_x+\eta_x)\Big]\nonumber\\
&=1+\sum_{k=1}^{|\Omega_\delta|}\sum_{\{x_1, \ldots, x_k\}\subset \Omega_{\delta} \atop x_i\neq x_j \mbox{\scriptsize\  for } i\neq j}\mathrm{E}_{\Omega_{\delta}}^{\mathrm{ref}}\Big[\prod_{i=1}^k(\lambda\omega_{x_i}\sigma_{x_i}+\eta_{x_i})\Big] \notag \\
& =1+\sum_{k=1}^{M}\sum_{\{x_1, \ldots, x_k\}\subset \Omega_{\delta} \atop x_i\neq x_j \mbox{\scriptsize\  for } i\neq j}\mathrm{E}_{\Omega_{\delta}}^{\mathrm{ref}}\Big[\prod_{i=1}^k(\lambda\omega_{x_i}\sigma_{x_i}+\eta_{x_i})\Big]+R_{M, \delta}, \label{truncated sum}
\end{align}
where we have truncated the sum at index $k=M\in \N$, and $R_{M, \delta}$ denotes the remainder. We will show in Section \ref{S:Truncate}
that this truncated sum converges to the series in \eqref{eq:MCalZ} truncated at $k=M$, i.e.,
\begin{equation} \label{eq:Lim1}
\sum_{k=1}^{M}\sum_{\{x_1, \ldots, x_k\}\subset \Omega_{\delta} \atop x_i\neq x_j \mbox{\scriptsize\  for } i\neq j} \!\!\!\!\!\! \mathrm{E}_{\Omega_{\delta}}^{\mathrm{ref}}\Big[\prod_{i=1}^k(\lambda\omega_{x_i}\sigma_{x_i}+\eta_{x_i})\Big]
\Aston{\delta} \sum_{k=1}^M \frac{\hat\lambda^k}{k!} \idotsint_{\Omega^k} \psi_0(x_1, \ldots, x_k) \, {\rm d}W(x_1) \ldots {\rm d}W(x_k),
\end{equation}
and we will show in Section \ref{S:Remainder} that
\begin{equation} \label{eq:Lim2}
\lim_{M\to\infty} \limsup_{\delta\downarrow 0} \E[R^2_{M, \delta}] = 0.
\end{equation}
Since assumptions {\bf (A1)-(A2)} imply that the series in \eqref{eq:MCalZ} is convergent in $L^2$, the conclusion of Theorem \ref{T:main}
follow immediately from \eqref{eq:Lim1} and \eqref{eq:Lim2}.
\qed

\subsection{Convergence of the truncated sum} \label{S:Truncate}
In this subsection, we verify \eqref{eq:Lim1}. In the truncated sum in \eqref{truncated sum}, for each $1\leq k\leq M$, we can further decompose
the $k$-th term into
\begin{equation}\label{4.5}
\begin{aligned}
\sum_{\{x_1, \ldots, x_k\}\subset \Omega_{\delta} \atop x_i\neq x_j \mbox{\scriptsize\  for } i\neq j} \!\!\!\!\!\!
\mathrm{E}_{\Omega_{\delta}}^{\mathrm{ref}}\Big[\prod_{i=1}^k(\lambda\omega_{x_i}\sigma_{x_i}+\eta_{x_i})\Big]
&= \!\!\!\!\!\! \sum_{\{x_1, \ldots, x_k\}\subset \Omega_{\delta}\atop x_i\neq x_j \mbox{\scriptsize\  for } i\neq j}
\!\!\!\!\!\! \lambda^k \mathrm{E}_{\Omega_{\delta}}^{\mathrm{ref}}[\sigma_{x_1} \ldots \sigma_{x_k}]\prod_{i=1}^k \omega_{x_i}\\
& \ \ \
+ \!\!\!\!\!\! \underbrace{\sum_{\{x_1, \ldots, x_k\}\subset \Omega_{\delta}\atop x_i\neq x_j \mbox{\scriptsize\  for } i\neq j}
\sum_{(I, J)\vdash \{x_1, \ldots, x_k\} \atop I\neq \emptyset} \!\!\!\!\!
\mathrm{E}_{\Omega_{\delta}}^{\mathrm{ref}}\Big[\prod_{u \in I}\eta_u\prod_{v \in J}\lambda\omega_v\sigma_v\Big]}_{\Ei_{k, \delta}},
\end{aligned}
\end{equation}
where $(I, J) \vdash \{x_1, \ldots, x_k\}$ denotes a partition with $I\cup J=\{x_1, \ldots, x_k\}$ and $I \cap J=\emptyset$.

The first term in \eqref{4.5} is in fact the dominant term. Indeed, if we sum it over $1\leq k\leq M$, then by the choice
of $\lambda=\hat \lambda \delta^{\frac{d}{2}-\gamma}$ and the definition of $\psi_{\delta}$ in \eqref{eq:psidelta}, we have
\begin{align*}
\sum_{k=1}^M  \sum_{\{x_1, \ldots, x_k\}\subset \Omega_{\delta}\atop x_i\neq x_j \mbox{\scriptsize\  for } i\neq j}
\!\!\!\!\!\! \lambda^k \mathrm{E}_{\Omega_{\delta}}^{\mathrm{ref}}[\sigma_{x_1} \ldots \sigma_{x_k}]\prod_{i=1}^k \omega_{x_i}
& = \sum_{k=1}^M \hat\lambda^k \!\!\!\!\! \sum_{\{x_1, \ldots, x_k\}\subset \Omega_{\delta}\atop x_i\neq x_j \mbox{\scriptsize\  for } i\neq j}
\!\!\!\!\!\! \psi_{\delta}(x_1, \ldots, x_k)\prod_{i=1}^k(\delta^{\frac{d}{2}}\omega_{x_i}) \\
& \Aston{\delta} \sum_{k=1}^M \frac{\hat\lambda^k}{k!} \idotsint_{\Omega^k} \psi_0(x_1, \ldots, x_k) \, {\rm d}W(x_1) \ldots {\rm d}W(x_k),
\end{align*}
where the convergence follows from assumption {\bf (A1)} and \cite[Theorem 2.3]{CSZ17a}.

Therefore to prove \eqref{eq:Lim1}, it only remains to show that for each $k\in\N$, the second term in \eqref{4.5}
satisfies
\begin{align}\label{4.8}
\Ei_{k, \delta} := \sum_{\{x_1, \ldots, x_k\}\subset \Omega_{\delta}\atop x_i\neq x_j \mbox{\scriptsize\  for } i\neq j}
\sum_{(I, J)\vdash \{x_1, \ldots, x_k\} \atop I\neq \emptyset} \!\!\!\!\!
\mathrm{E}_{\Omega_{\delta}}^{\mathrm{ref}}\Big[\prod_{u \in I}\eta_u\prod_{v \in J}\lambda \sigma_v \omega_v\Big] \xrightarrow[\delta\downarrow 0]{}  0 \ \ \  \rm{in \ probability}.
\end{align}
To prove this, we first observe that for two different sets $\{x_1, \ldots, x_k\}\neq \{\tilde x_1, \ldots, \tilde x_k\}$, the corresponding
summands in \eqref{4.8} are $L^2$-orthogonal to each other.

\begin{lemma}\label{L:ortho}
Let $I, J, \tilde I, \tilde J \subset \Omega_\delta$ satisfy $I \cap J=\emptyset$, $\tilde I \cap \tilde J=\emptyset$, and
$I \cup J\neq \tilde I \cup \tilde J$.  Then,
\begin{align}
\mathbb{E}\left[\mathrm{E}_{\Omega_{\delta}}^{\mathrm{ref}}\left[\prod_{u \in I}\eta_{u} \prod_{v \in J}\lambda\sigma_v\omega_v\right]\mathrm{E}_{\Omega_{\delta}}^{\mathrm{ref}}\Bigg[\prod_{u \in \tilde{I}}\eta_{u}\prod_{v \in \tilde{J}}\lambda\sigma_v\omega_v\Bigg]\right]=0.\nonumber
\end{align}
\end{lemma}

\begin{proof}
We can rewrite the expectation as
\begin{align}
&\mathrm{E}_{\Omega_{\delta}}^{\mathrm{ref},\otimes 2}\left[\mathbb{E}\left[\prod_{u \in I}\eta_{u}(\sigma, \omega) \prod_{v \in J}\lambda\sigma_v\omega_v\prod_{\tilde u \in \tilde{I}}\eta_{\tilde u}(\sigma', \omega)\prod_{\tilde v \in \tilde{J}}\lambda \sigma'_{\tilde v}\omega_{\tilde v}\right]\right], \label{eq:ortho}
\end{align}
where $\mathrm{E}_{\Omega_{\delta}}^{\mathrm{ref},\otimes 2}$ denotes the expectation with respect to $\sigma$ and $\sigma'$, two independent spin configurations with law $\p_{\Omega_{\delta}}^{\mathrm{ref}}$. Recall that conditional on $\sigma$ and
$\sigma'$,
$$
\eta_x(\sigma, \omega)= e^{\lambda\omega_x\sigma_x-\phi(\lambda\sigma_x)}-1-\lambda\omega_x\sigma_x
$$
and $\eta_x(\sigma', \omega)$ depend only on $\omega_x$. Therefore they are independent of $(\omega_y, \eta_y(\sigma, \omega), \eta_y(\sigma', \omega))$ for $y\neq x$. The assumption $I \cup J\neq \tilde I \cup \tilde J$ implies that there is some site $x\in I\cup J\cup \tilde I \cup \tilde J$ which appears exactly once in the product in \eqref{eq:ortho}. By the independence of $(\omega_y, \eta_y(\sigma, \omega), \eta_y(\sigma', \omega))$ for different $y\in \Omega_\delta$ and the fact that $\E[\eta_y(\sigma, \omega)]=\E[\eta_y(\sigma', \omega)]=\E[\omega_y]=0$, it follows easily
that the expectation in \eqref{eq:ortho} equals $0$.
\end{proof}

Applying Lemma \ref{L:ortho}, to prove \eqref{4.8}, it then suffices to show that
\begin{align}
\label{4.10}
\E[\Ei_{k, \delta}^2] =
\sum_{\{x_1, \ldots, x_k\}\subset \Omega_{\delta}\atop x_i\neq x_j \mbox{\scriptsize\  for } i\neq j}
\E\left[\left(\sum_{(I, J)\vdash \{x_1, \ldots, x_k\} \atop I\neq \emptyset} \!\!\!\!\!
\mathrm{E}_{\Omega_{\delta}}^{\mathrm{ref}}\Big[\prod_{u \in I}\eta_u\prod_{v \in J}\lambda \sigma_v \omega_v\Big]\right)^2\right]
\xrightarrow[\delta\downarrow 0]{}  0 .
\end{align}
Note that the sum over $(I, J)\vdash \{x_1, \ldots, x_k\}$, $I\neq \emptyset$, contains $2^k-1$ terms. Using $(\sum_{i=1}^{n}a_i)^2\leq n\sum_{i=1}^{n}a_i^2$, we can bound
\begin{align}
\E[\Ei_{k, \delta}^2] & \leq 2^k
\sum_{\{x_1, \ldots, x_k\}\subset \Omega_{\delta}\atop x_i\neq x_j \mbox{\scriptsize\  for } i\neq j}
\sum_{(I, J)\vdash \{x_1, \ldots, x_k\} \atop I\neq \emptyset} \!\!\!\!\!
\E\left[ \mathrm{E}_{\Omega_{\delta}}^{\mathrm{ref}}\Big[\prod_{u \in I}\eta_u\prod_{v \in J}\lambda \sigma_v \omega_v\Big]^2 \right] \notag \\
& = 2^k \sum_{\iota =1}^k \sum_{I, J \subset \Omega_\delta, I\cap J=\emptyset \atop |I|=\iota, |J|=k-\iota} \E\left[ \mathrm{E}_{\Omega_{\delta}}^{\mathrm{ref}}\Big[\prod_{u \in I}\eta_u\prod_{v \in J}\lambda \sigma_v \omega_v\Big]^2 \right] \notag \\
& = 2^k \sum_{\iota =1}^k \sum_{I, J \subset \Omega_\delta, I\cap J=\emptyset \atop |I|=\iota, |J|=k-\iota} \mathrm{E}_{\Omega_{\delta}}^{\mathrm{ref}, \otimes 2}\left[ \prod_{u \in I} \E[\eta_u(\sigma, \omega) \eta_u(\sigma', \omega)] \prod_{v \in J}\lambda^2 \sigma_v\sigma'_v\right], \label{2.8}
\end{align}
where $\sigma'$ is an independent copy of $\sigma$ with law $\p_{\Omega_{\delta}}^{\mathrm{ref}}$. Therefore to prove \eqref{4.10}, it suffices to show that for each $1\leq \iota \leq k$,
\begin{align}\label{4.11}
\sum_{I, J \subset \Omega_\delta, I\cap J=\emptyset \atop |I|=\iota, |J|=k-\iota}  \mathrm{E}_{\Omega_{\delta}}^{\mathrm{ref},\otimes 2}\left[\prod_{u \in I}\mathbb{E}[\eta_{u}(\sigma, \omega)\eta_{u}(\sigma',\omega)]\prod_{v \in J}\lambda^2\sigma_v\sigma_v'\right]
\xrightarrow[\delta\downarrow 0]{}  0.
\end{align}

To see heuristically why \eqref{4.11} holds, recall from the discussion leading to \eqref{eq:wtZ3} that $\lambda=\lambda_\delta=\hat \lambda \delta^{\frac{d}{2}-\gamma}$ is chosen such that if each spin $\sigma_u$ is matched with a factor of $\lambda$ and $u$ is summed over $\Omega_\delta$, then the spin correlations $\mathrm{E}_{\Omega_{\delta}}^{\mathrm{ref}}[\sigma_{x_1^{\delta}} \ldots \sigma_{x_k^{\delta}}]$ would be properly
normalised and we will have convergence as in \eqref{eq:wtZ3}. However, if some of the spins among $\sigma_{x_i^{\delta}}$ coincide, then by assumption
{\bf (A3)}, assigning one factor of $\lambda$ to each spin gives more powers of $\lambda$ than needed to normalize $\mathrm{E}_{\Omega_{\delta}}^{\mathrm{ref}}[\sigma_{x_1^{\delta}} \ldots \sigma_{x_k^{\delta}}]$.
This is exactly what happens when we expand $\E[\eta_{u}(\sigma, \omega)\eta_{u}(\sigma',\omega)]$ in powers of $\lambda$, $\sigma_u$ and $\sigma_u'$.
We will perform this expansion next, which is a bit involved.

Recall that $\eta_x(\sigma, \omega)= e^{\lambda\omega_x\sigma_x-\phi(\lambda\sigma_x)}-1-\lambda\omega_x\sigma_x$ and
$\phi(a) = \log \E[e^{a \omega_x}]$. Therefore
\begin{align}\label{4.12}
\mathbb{E}[\eta_{u}(\sigma, \omega)\eta_{u}(\sigma',\omega)]
&=\mathbb{E}[(e^{\lambda\omega_u\sigma_u-\phi(\lambda \sigma_u)}-1-\lambda\omega_u\sigma_u)(e^{\lambda\omega_u\sigma'_u-\phi(\lambda \sigma'_u)}-1-\lambda\omega_u\sigma'_u)]\nonumber\\
&=e^{\phi(\lambda (\sigma_u+\sigma'_u))-\phi(\lambda \sigma_u)-\phi(\lambda \sigma'_u)}\nonumber
-1+\lambda^2\sigma_u\sigma'_u \nonumber\\
&\qquad -\lambda\sigma_u\mathbb{E}[\omega_u e^{\lambda\omega_u\sigma'_u-\phi(\lambda \sigma'_u)}]-\lambda\sigma'_u\mathbb{E}[\omega_u e^{\lambda\omega_u\sigma_u-\phi(\lambda \sigma_u)}].
\end{align}
Since $\phi(a)$ is assumed to be finite for all $|a|<a_0$ for some $a_0>0$ and $\phi(0)=0$, $\phi$ must be analytic on the ball $|z|<z_0$ for some $z_0>0$ and has power series expansion
\begin{equation} \label{4.13}
\phi(z)=\sum_{m=1}^{\infty}\frac{\kappa_m}{m!}z^m,
\end{equation}
where $\kappa_m=\frac{{\rm d}^{m}\phi(z)}{{\rm d}z^{m}}\Bigr|_{z=0}$ is the $m$-th cumulant of $\omega_x$ with $\kappa_1=\E[\omega_x]=0$
and $\kappa_2=\mathbb{V}{\rm ar}(\omega_x^2)=1$. Since $|\sigma_u|\leq K$ by assumption {\bf (A0)}, for $\lambda=\lambda_\delta$ sufficiently
small, we can rewrite the last term in \eqref{4.12} (similarly for the second last term) as
\begin{align}
\label{4.14}  \lambda\sigma'_u\mathbb{E}[\omega_u e^{\lambda\omega_u\sigma_u-\phi(\lambda \sigma_u)}]
&=\lambda\sigma'_u\frac{{\rm d}\phi(z)}{{\rm d}z}\Bigr|_{z=\lambda\sigma_u}
=\lambda^2\sigma_u\sigma'_u+\sum_{m=3}^{\infty}\frac{\kappa_m \lambda^m}{(m-1)!}(\sigma_u)^{m-1}\sigma'_u.
\end{align}
Substituting into  \eqref{4.12} then yields
\begin{align}\label{4.12a}
\mathbb{E}[\eta_{u}(\sigma, \omega)\eta_{u}(\sigma',\omega)]
&=e^{\phi(\lambda (\sigma_u+\sigma'_u))-\phi(\lambda \sigma_u)-\phi(\lambda \sigma'_u)}
-1-\lambda^2\sigma_u\sigma'_u \nonumber\\
&\ \ -\sum_{m=3}^{\infty}\frac{\kappa_m \lambda^m}{(m-1)!}(\sigma'_u)^{m-1}\sigma_u - \sum_{m=3}^{\infty}\frac{\kappa_m \lambda^m}{(m-1)!}(\sigma_u)^{m-1}\sigma'_u.
\end{align}
Using \eqref{4.13} and the binomial expansion yields
\begin{align}\label{4.15}
\phi(\lambda (\sigma_u+\sigma'_u))-\phi(\lambda \sigma_u)-\phi(\lambda \sigma'_u)=\sum_{m=2}^{\infty}\frac{\kappa_m}{m!}\lambda^{m}\sum_{l=1}^{m-1}\binom{m}{l}
(\sigma_u)^{m-l}(\sigma'_u)^{l}.
\end{align}
Again by Taylor expansion, for $\lambda=\lambda_\delta$ sufficiently small, we have
\begin{equation}\label{4.16}
\begin{aligned}
&e^{\phi(\lambda (\sigma_u+\sigma'_u))-\phi(\lambda \sigma_u)-\phi(\lambda \sigma'_u)}-1 \\
=\ & \sum_{m=2}^{\infty}\frac{\kappa_m}{m!}\lambda^{m}\sum_{l=1}^{m-1}\binom{m}{l}
(\sigma_u)^{m-l}(\sigma'_u)^{l}+\sum_{j=2}^{\infty}\frac{1}{j!}\left[\sum_{m=2}^{\infty}\frac{\kappa_m}{m!}\lambda^{m}\sum_{l=1}^{m-1}\binom{m}{l}
(\sigma_u)^{m-l}(\sigma'_u)^{l}\right]^j.
\end{aligned}
\end{equation}
Since $\kappa_2=1$, we can decompose the first term in \eqref{4.16} by separating the contributions from $m=2$, $m=3$ and $l\in \{1, m-1\}$,
versus $m\geq 4$ and $l\in \{2, \ldots, m-2\}$ to rewrite it as
\begin{align*}
&\lambda^2\sigma_u\sigma'_u+\sum_{m=3}^{\infty}\frac{\kappa_m\lambda^{m}}{(m-1)!}
\Big((\sigma_u)^{m-1}\sigma'_u+ \sigma_u(\sigma'_u)^{m-1}\Big)
+\sum_{m=4}^{\infty}\frac{\kappa_m\lambda^{m}}{m!}\sum_{l=2}^{m-2}\binom{m}{l}
(\sigma_u)^{m-l}(\sigma'_u)^{l}.
\end{align*}
Substituting into \eqref{4.16} and then \eqref{4.12a} gives
\begin{equation}\label{4.18}
\begin{aligned}
\mathbb{E}[\eta_{u}(\sigma, \omega)\eta_{u}(\sigma',\omega)]& =\sum_{m=4}^{\infty}\frac{\kappa_m\lambda^{m}}{m!}\sum_{l=2}^{m-2}\binom{m}{l}
(\sigma_u)^{m-l}(\sigma'_u)^{l} \\
& \quad +\sum_{j=2}^{\infty}\frac{1}{j!}\left[\sum_{m=2}^{\infty}\frac{\kappa_m \lambda^{m}}{m!}\sum_{l=1}^{m-1}\binom{m}{l}
(\sigma_u)^{m-l}(\sigma'_u)^{l}\right]^j.
\end{aligned}
\end{equation}
Next we will expand the second term in the r.h.s.\ of \eqref{4.18} and perform resummation to write
\begin{align}\label{4.21}
\mathbb{E}[\eta_{u}(\sigma, \omega)\eta_{u}(\sigma',\omega)]
&=\sum_{m=4}^{\infty}\sum_{l=2}^{m-2}\lambda^{m}a_{m,l}
(\sigma_u)^{m-l}(\sigma'_u)^{l},
\end{align}
where $a_{m,l}$ are constants to be determined later. To justify the resummation, we need to show absolute summability.
For this, we first need to bound $\frac{|\kappa_m|}{m!}$. Since the power series for $\phi$ in \eqref{4.13} has a
positive radius of convergence, there exists a finite constant $\mathcal{C}>0$ such that
\begin{align}\label{4.19}
\frac{|\kappa_m|}{m!}\leq \mathcal{C}^m\ \ \mbox{for any } m\geq 2.
\end{align}
Combined with the assumption {\bf (A0)} that $|\sigma_u|\leq K$ for some finite $K$, we can bound
$$
\sum_{m=2}^{\infty}\sum_{l=1}^{m-1}\left|\frac{\kappa_m \lambda^{m}}{m!}\binom{m}{l}
(\sigma_u)^{m-l}(\sigma'_u)^{l}\right|\leq \sum_{m=2}^{\infty}
\mathcal{C}^{m} \lambda^{m}K^m\sum_{l=1}^{m-1}\binom{m}{l}\leq \sum_{m=2}^{\infty} (2\lambda K\mathcal{C})^{m},
$$
which is convergent when $\lambda=\lambda_\delta=\hat \lambda \delta^{\frac{d}{2}-\gamma}$ is small enough.
This implies that the r.h.s.\ of \eqref{4.18} is absolutely convergent, and we can rearrange terms into a power
series in $\lambda$ as in \eqref{4.21}. To identify the constants $a_{m, l}$, for $j\geq 2$, we expand
\begin{align}\label{4.22}
& \sum_{j=2}^{\infty}\frac{1}{j!} \left[\sum_{m=2}^{\infty}\frac{\kappa_m \lambda^{m}}{m!}\sum_{l=1}^{m-1}\binom{m}{l}
(\sigma_u)^{m-l}(\sigma'_u)^{l}\right]^j \notag \\
=\ & \sum_{j=2}^{\infty}\frac{1}{j!}  \sum_{m_1,\cdots, m_j\geq2}\sum_{\substack{1\leq l_i\leq m_i-1, \notag\\
i=1,\cdots,j}}
\prod_{i=1}^j\left(\frac{\kappa_{m_i} \lambda^{m_i}}{m_i!} \binom{m_i}{l_i}
(\sigma_u)^{(m_i-l_i)}(\sigma'_u)^{l_i}\right) \nonumber\\
=:& \sum_{j=2}^{\infty}\frac{1}{j!} \sum_{m=2j}^{\infty}\sum_{l=j}^{m-j}
c_{j,m,l}\lambda^{m}(\sigma_u)^{m-l}(\sigma'_u)^{l} \notag \\
=: & \sum_{m=4}^{\infty}\lambda^{m}\sum_{l=2}^{m-2}d_{m,l}
(\sigma_u)^{m-l}(\sigma'_u)^{l},
\end{align}
where for $j\geq 2$, $m\geq 2j$, and $j\leq l\leq m-j$,
\begin{align}\label{4.23a}
c_{j, m, l}& :=\sum_{\substack{m_1,\cdots, m_j\geq2\\
m_1+\cdots+ m_j=m}
}\sum_{\substack{1\leq l_i\leq m_i-1,\\
i=1,\cdots,j\\
l_1+\cdots+l_j=l}}
\prod_{i=1}^j\frac{\kappa_{m_i}}{m_i!} \binom{m_i}{l_i},
\end{align}
and
\begin{eqnarray*}
d_{m,l}=\begin{cases} \sum\limits_{j=2}^{l}\frac{1}{j!}c_{j,m,l}, &  \mbox{if $2\leq l\leq \lfloor \frac{m}{2} \rfloor$},\\
    d_{m, m-l}, &  \mbox{if $\lfloor \frac{m}{2} \rfloor< l\leq m-2$}.
    \end{cases}
\end{eqnarray*}
Combined with \eqref{4.18}, it follows that \eqref{4.21} holds with
\begin{align}\label{4.24}
    a_{m,l}:=\frac{\kappa_m}{m!}\binom{m}{l}+d_{m,l} = a_{m, m-l}.
\end{align}
This concludes the expansion of $\mathbb{E}[\eta_{u}(\sigma, \omega)\eta_{u}(\sigma',\omega)]$ stated in \eqref{4.21}.
\bigskip

Substituting \eqref{4.21} into \eqref{4.11}, our goal is to show
\begin{align}
\label{4.26}
\sum_{I, J \subset \Omega_\delta, I\cap J=\emptyset \atop |I|=\iota, |J|=k-\iota}  \mathrm{E}_{\Omega_{\delta}}^{\mathrm{ref},\otimes 2}\left[\prod_{u \in I}\left(\sum_{m=4}^{\infty}\sum_{l=2}^{m-2}\lambda^{m}a_{m,l}
(\sigma_u)^{m-l}(\sigma'_u)^{l}\right)\prod_{v \in J}\lambda^2\sigma_v\sigma_v'\right]
\xrightarrow[\delta\downarrow 0]{}  0.
\end{align}
To control the convergence of the series, first recall from \eqref{4.19} that $\frac{|\kappa_m|}{m!}\leq \mathcal{C}^m$. Therefore by \eqref{4.24}, for $2\leq l\leq \lfloor \frac{m}{2} \rfloor$, we have
\begin{align}\label{4.27}
 |a_{m,l}|=|a_{m, m-l}| \leq \mathcal{C}^{m}\binom{m}{l}+\sum\limits_{j=2}^{l}\frac{1}{j!}|c_{j,m,l}|.
\end{align}
Recalling the definition of $c_{j, m,l}$ from \eqref{4.23a}, we have
$$
|c_{j,m,l}|\leq \mathcal{C}^{m}\sum_{\substack{m_1,\cdots,m_j\geq2\\
m_1+\cdots+m_j=m}
}\sum_{\substack{1\leq l_i\leq m_i-1,\\
i=1,\cdots,j }}
\prod_{i=1}^j\binom{m_i}{l_i} \leq \mathcal{C}^{m}\sum_{\substack{m_1,\cdots, m_j\geq2\\
m_1+\cdots+m_j=m}
}\prod_{i=1}^j 2^{m_i}
= (2\mathcal{C})^{m}\binom{m-j-1}{j-1},
$$
where the last identity holds by elementary combinatorial considerations. Substituting this into \eqref{4.27} then yields that,
for $2\leq l\leq \lfloor \frac{i}{2} \rfloor$,
$$
 |a_{m,l}| = |a_{m, m-l}| \leq \mathcal{C}^{m}\binom{m}{l}+(2\mathcal{C})^{m}\sum\limits_{j=2}^{l}\frac{1}{j!}\binom{m-j-1}{j-1}\leq (2\mathcal{C})^{m}+(2\mathcal{C})^{m}\sum_{j=2}^{l}\frac{1}{j!}2^{m-j-1} \leq 2\cdot (4\Ci)^m.
$$
Therefore
\begin{align} \label{4.28}
\sum\limits_{l=2}^{m-2}|a_{m,l}|\leq 2m \cdot (4\Ci)^m \leq (8\Ci)^m.
\end{align}
Applying this bound and the assumption {\bf (A0)} that $|\sigma_u|\leq K$ for any $u\in\Omega_{\delta}$, the expectation in \eqref{4.26} can be
bounded from above by
\begin{align}\label{4.29}
\mathrm{E}_{\Omega_{\delta}}^{\mathrm{ref},\otimes 2}\left[\prod_{u \in I}\left(\sum_{m=4}^{\infty}\sum_{l=2}^{m-2}\left|\lambda^{m}a_{m,l}
(\sigma_u)^{m-l}(\sigma'_u)^{l}\right|\right)\prod_{v \in J}\left|\lambda^2\sigma_v\sigma_v'\right|\right]\leq  \left(\sum_{m=4}^{\infty}(8\mathcal{C}\lambda K)^{m}\right)^{|I|}(\lambda K)^{2|J|},
\end{align}
which is finite for $\lambda=\lambda_\delta:=\hat \lambda \delta^{\frac{d}{2}-\gamma}$ sufficiently small.
It follows that the expectation in  \eqref{4.26} is absolutely convergent and can be expanded as
\begin{align}
&\mathrm{E}_{\Omega_{\delta}}^{\mathrm{ref},\otimes 2}\left[\prod_{u \in I}\left(\sum_{m=4}^{\infty}\sum_{l=2}^{m-2}\lambda^{m}a_{m,l}
(\sigma_u)^{m-l}(\sigma'_u)^{l}\right)\prod_{v \in J}\lambda^2\sigma_v\sigma_v'\right]  \notag \\
&=\mathrm{E}_{\Omega_{\delta}}^{\mathrm{ref},\otimes 2}\left[\sum_{m_1,\cdots, m_{|I|}\geq 4}\sum_{\substack{2\leq l_i\leq m_i-2,\\
i=1,\cdots,|I|}}\prod_{i=1}^{|I|}\left(\lambda^{m_i} a_{m_i, l_i}
(\sigma_{u_i})^{m_i-l_i}(\sigma'_{u_i})^{l_i}\right)\prod_{v \in J}\lambda^2\sigma_v\sigma_v'\right]\nonumber\\
&= \sum_{m_1,\cdots, m_{|I|}\geq4}\sum_{\substack{2\leq l_i\leq m_i-2,\\
i=1,\cdots,|I|}}\mathrm{E}_{\Omega_{\delta}}^{\mathrm{ref},\otimes 2}\left[\prod_{i=1}^{|I|}\left(\lambda^{m_i}a_{m_i, l_i}
(\sigma_{u_i})^{m_i-l_i}(\sigma'_{u_i})^{l_i}\right)\prod_{v \in J}\lambda^2\sigma_v\sigma_v'\right], \label{4.30}
\end{align}
where we have assumed that $I=\{u_1,\cdots, u_{|I|}\}$. To show \eqref{4.26}, it suffices to show
\begin{align} \label{4.30a}
\sum_{I, J \subset \Omega_\delta, I\cap J=\emptyset \atop |I|=\iota, |J|=k-\iota}
\sum_{\substack{m_1,\cdots,m_{|I|}\geq4\\ 2\leq l_i\leq m_i-2,\\ i=1,\cdots,|I|}}
\left|\mathrm{E}_{\Omega_{\delta}}^{\mathrm{ref},\otimes 2}\Bigg[\prod_{i=1}^{|I|}\left(\lambda^{m_i}a_{m_i, l_i}
(\sigma_{u_i})^{m_i-l_i}(\sigma'_{u_i})^{l_i}\right)\prod_{v \in J}\lambda^2\sigma_v\sigma_v'\Bigg]\right|
\xrightarrow[\delta\downarrow 0]{}  0.
\end{align}
The absolute value in the left hand side of \eqref{4.30a} can be bounded by
\begin{align}\label{4.31}
\lambda^{2|J|}\left(\prod_{i=1}^{|I|}\lambda^{m_i}|a_{m_i, l_i}|\right)\left|\mathrm{E}_{\Omega_{\delta}}^{\mathrm{ref}}\left[\prod_{i=1}^{|I|}
(\sigma_{u_i})^{m_i-l_i}\prod_{v \in J}\sigma_v\right]\right|\left|\mathrm{E}_{\Omega_{\delta}}^{\mathrm{ref}}\left[\prod_{i=1}^{|I|}
(\sigma'_{u_i})^{l_i}\prod_{v \in J}\sigma_v'\right]\right|.
\end{align}
Using assumption {\bf (A3)}, where $(r)_2 := r \mbox{ (mod 2)}$, there is a universal constant $C\geq 1$ such that
\begin{align}
\left|\mathrm{E}_{\Omega_{\delta}}^{\mathrm{ref}}\Bigg[\prod_{i=1}^{|I|}
(\sigma_{u_i})^{m_i-l_i}\prod_{v \in J}\sigma_v\Bigg]\right|
&\leq C^{\sum\limits_{i=1}^{|I|}[m_i-l_i-(m_i-l_i)_2]}\left|\e_{\Omega_{\delta}}^{\mathrm{ref}}\Bigg[\prod_{i=1}^{|I|}\sigma_{u_i}^{(m_i-l_i)_2}\prod_{v \in J}\sigma_v\Bigg]\right|\nonumber
\end{align}
and
\begin{align}
\left|\mathrm{E}_{\Omega_{\delta}}^{\mathrm{ref}}\Bigg[\prod\limits_{i=1}^{|I|}
(\sigma'_{u_i})^{l_i}\prod\limits_{v \in J}\sigma_v'\Bigg]\right|
&\leq C^{\sum\limits_{i=1}^{|I|}[l_i-(l_i)_2]}\left|\e_{\Omega_{\delta}}^{\mathrm{ref}}\Bigg[\prod_{i=1}^{|I|}(\sigma'_{u_i})^{(l_i)_2}\prod_{v \in J}\sigma'_v\Bigg]\right|.\nonumber
\end{align}
Let
\begin{equation}\label{eq:I1}
\begin{aligned}
I_1 & =\{u_i\in I: (m_i-l_i)_2 =1, 1\leq i \leq |I|\}, \\
I'_1 & =\{u_i\in I: (l_i)_2=1, 1\leq i\leq |I|\}.
\end{aligned}
\end{equation}
Then we can write
\begin{align}\label{4.32}
&\left|\mathrm{E}_{\Omega_{\delta}}^{\mathrm{ref}}\left[\prod_{i=1}^{|I|}
(\sigma_{u_i})^{m_i-l_i}\prod_{v \in J}\sigma_v\right]\right|\left|\mathrm{E}_{\Omega_{\delta}}^{\mathrm{ref}}\left[\prod_{i=1}^{|I|}
(\sigma'_{u_i})^{l_i}\prod_{v \in J}\sigma_v'\right]\right|\nonumber\\
\leq\ & c_{|I|}C^{\sum\limits_{i=1}^{|I|}m_i}\left|\e_{\Omega_{\delta}}^{\mathrm{ref}}\Big[\prod_{u \in I_1}\sigma_u\prod_{v \in J}\sigma_v\Big] \right|
\left|\e_{\Omega_{\delta}}^{\mathrm{ref}}\Big[\prod_{u \in I'_1}\sigma'_u\prod_{v \in J}\sigma'_v\Big] \right|\nonumber\\
\leq\ & c_{|I|} C^{\sum\limits_{i=1}^{|I|} m_i}\left(\e_{\Omega_{\delta}}^{\mathrm{ref}}\Big[\prod_{u \in I_1}\sigma_u\prod_{v \in J}\sigma_v\Big]^2+
\e_{\Omega_{\delta}}^{\mathrm{ref}}\Big[\prod_{u \in I'_1}\sigma'_u\prod_{v \in J}\sigma'_v\Big]^2\right),
\end{align}
where
$c_{|I|}=C^{-\sum\limits_{i=1}^{|I|}[(m_i-l_i)_2 +(l_i)_2]} \leq 1$. Inserting this bound into \eqref{4.31}, the absolute value in the left hand side of \eqref{4.30a} is bounded from above by
\begin{align*}
c_{|I|}\lambda^{2|J|}\left(\prod_{i=1}^{|I|}(\lambda C)^{m_i}|a_{m_i,l_i}|\right)
\left(\e_{\Omega_{\delta}}^{\mathrm{ref}}\Big[\prod_{u \in I_1}\sigma_u\prod_{v \in J}\sigma_v\Big]^2+
\e_{\Omega_{\delta}}^{\mathrm{ref}}\Big[\prod_{u \in I'_1}\sigma'_u\prod_{v \in J}\sigma'_v\Big]^2\right).
\end{align*}
Substitute this bound into \eqref{4.30a}. By the symmetry between $(\sigma, m_i-l_i, I_1)$ and $(\sigma', l_i, I_1')$,
the two terms in the sum above give equal contribution, and hence it suffices to show that
\begin{align}\label{2.30}
\sum_{\substack{I, J \subset \Omega_\delta, I\cap J=\emptyset \\ |I|=\iota, |J|=k-\iota\\ I_1\subset I}}
\sum_{\substack{m_1,\cdots,m_{|I|}\geq4\\ 2\leq l_i\leq m_i-2,\\ i=1,\cdots,|I|}}
\lambda^{2|J|} \Bigg(\prod_{i=1}^{|I|}(\lambda C)^{m_i}|a_{m_i,l_i}|\Bigg) \e_{\Omega_{\delta}}^{\mathrm{ref}}\left[\prod_{u \in I_1}\sigma_u\prod_{v \in J}\sigma_v\right]^2\xrightarrow[\delta\downarrow 0]{}  0.
\end{align}
We can further decompose the sum above according to $\iota_1:= |I_1|$. Denote $I_2:=I \backslash I_1$, $\iota_2:=|I_2|$ and $\iota_3=|J|$. Then it suffices to show that for any $\iota_1, \iota_2, \iota_3\geq 0$ with $\iota:=\iota_1+\iota_2 \geq 1$ and $\iota_1+\iota_2+\iota_3=k$,
\begin{align}\label{4.35a}
\sum_{\substack{I_1, I_2, J \subset \Omega_\delta \\ I_1\cap I_2=\emptyset, I \cap J=\emptyset \\ |I_1|=\iota_1, |I_2|=\iota_2, |J|=\iota_3}}
\!\!\!\!\!\!\!\!\!\!\!\! \lambda^{2|J|}\left( \sum_{\substack{m_1,\cdots,m_{|I|}\geq4\\ 2\leq l_i\leq m_i-2,\\ i=1,\cdots,|I|}}
\Bigg(\prod_{i=1}^{|I|}(\lambda C)^{m_i}|a_{m_i,l_i}|\Bigg)\right) \e_{\Omega_{\delta}}^{\mathrm{ref}}\left[\prod_{u \in I_1}\sigma_u\prod_{v \in J}\sigma_v\right]^2\xrightarrow[\delta\downarrow 0]{}  0.
\end{align}
Using \eqref{4.28}, the above sum over $m_i$ and $l_i$ can be bounded by
\begin{equation}\label{2.33}
\begin{aligned}
\sum_{\substack{m_1,\cdots,m_{|I|}\geq4\\ 2\leq l_i\leq m_i-2,\\ i=1,\cdots,|I|}}
\Bigg(\prod_{i=1}^{|I|}(\lambda C)^{m_i}|a_{m_i,l_i}|\Bigg)
& = \prod_{i=1}^{|I|} \Big(\sum_{m_i=4}^\infty (\lambda C)^{m_i} \sum_{l_i=2}^{m_i-2} |a_{m_i,l_i}|\Big) \\
&\leq \prod_{i=1}^{|I|} \Big(\sum_{m_i=4}^\infty (8\lambda C \Ci)^{m_i}\Big) \leq (\hat{C}\lambda)^{4|I|},
\end{aligned}
\end{equation}
where $\hat{C}=16 C \Ci$ and the bound holds for $\lambda=\lambda_\delta=\hat \lambda \delta^{\frac{d}{2}-\gamma}\leq \hat C^{-1}$. The l.h.s.\ of \eqref{4.35a}
can then be bounded by
\begin{align}
& \hat C^{4|I|} (\lambda)^{4|I|+2|J|} \!\!\!\!\!\!\!\!\!\!\!\!\!\!\! \sum_{\substack{I_1, I_2, J \subset \Omega_\delta \\ I_1\cap I_2=\emptyset, I \cap J=\emptyset \\ |I_1|=\iota_1, |I_2|=\iota_2, |J|=\iota_3}}
\!\!\!\!\!\!\!\!\!\!\!\!   \e_{\Omega_{\delta}}^{\mathrm{ref}}\left[\prod_{u \in I_1}\sigma_u\prod_{v \in J}\sigma_v\right]^2 \notag \\
\leq\ &  \hat C^{4|I|} (\hat \lambda \delta^{\frac{d}{2}-\gamma})^{4|I|+2|J|}  (|\Omega| \delta^{-d})^{|I_2|} {k-|I_2| \choose |I_1|}
 \!\!\! \sum_{G\subset \Omega_\delta, |G|=k-\iota_2} \!\!\!\!\!\!\!\!\! \delta^{d |G|}
 \left(\delta^{-\gamma |G|} \e_{\Omega_{\delta}}^{\mathrm{ref}}\Big[\prod_{u \in G}\sigma_u \Big]\right)^2 \cdot \delta^{(2\gamma-d) (k- \iota_2)} \notag \\
\leq \ & \tilde C^k \delta^{d(\iota_1+\iota_2) -2\gamma(\iota_1 + 2\iota_2)} \!\!\!\!\!\!\!\!\!\! \sum_{G\subset \Omega_\delta, |G|=\iota_1+\iota_3} \!\!\!\!\!\! \delta^{d |G|} \left(\delta^{-\gamma |G|} \e_{\Omega_{\delta}}^{\mathrm{ref}}\Big[\prod_{u \in G}\sigma_u \Big]\right)^2 \notag \\
= \ & \tilde C^k \delta^{d(\iota_1+\iota_2) -2\gamma(\iota_1 + 2\iota_2)} \frac{\Vert \psi_\delta\Vert^2_{L^2(\Omega^{\iota_1+\iota_3})}}{(\iota_1+\iota_3)!},
\label{2.34}
\end{align}
where in the second line, $(|\Omega| \delta^{-d})^{|I_2|}$ bounds the number of choices of $I_2\subset \Omega_\delta$, $G:=I_1\cup J$, and the last
power of $\delta$ cancels out the powers of $\delta$ we inserted in the sum over $G\subset \Omega_\delta$. In the third line, $\tilde C^k$ is chosen to bound the factors $\hat C^{4|I|} \hat\lambda^{4|I|+2|J|} |\Omega|^{|I_2|} {k-|I_2| \choose |I_1|}$, using that ${k-|I_2| \choose |I_1|}\leq 2^k$.
Note that $\tilde C$ does not depend on $\delta$, $\iota_1$, $\iota_2$ or $\iota_3$, and $\Vert \psi_\delta\Vert^2_{L^2(\Omega^{\iota_1+\iota_3})}\to \Vert \psi_{0}\Vert_{L^{2}(\Omega^{\iota_1+\iota_3})}^{2}$ as $\delta\downarrow 0$ by assumption {\bf (A1)}. Therefore \eqref{4.35a} holds if
$$
d(\iota_1+\iota_2) -2\gamma(\iota_1 + 2\iota_2) >0 \quad \Longleftrightarrow \quad \gamma < \frac{d}{2} \cdot \frac{\iota_1+\iota_2}{\iota_1+2\iota_2},
$$
which holds for all $\iota_1, \iota_2\geq 0$ with $\iota_1+\iota_2\geq 1$ if and only if $\gamma <d/4$. This is part of the assumption in Theorem
\ref{T:main} and hence the proof of \eqref{4.10} and \eqref{4.8} is complete.
\qed

\subsection{Control of the remainder} \label{S:Remainder}

In this section, we conclude the proof of Theorem \ref{T:main} by proving \eqref{eq:Lim2}, that is, the remainder
\begin{align*}
R_{M, \delta}& = \sum_{k=M+1}^\infty \sum_{\{x_1, \ldots, x_k\}\subset \Omega_{\delta} \atop x_i\neq x_j \mbox{\scriptsize\  for } i\neq j} \!\!\!\!\!\!
\mathrm{E}_{\Omega_{\delta}}^{\mathrm{ref}}\Big[\prod_{i=1}^k(\lambda\omega_{x_i}\sigma_{x_i}+\eta_{x_i})\Big] \\
& =
\sum_{k=M+1}^\infty \sum_{\{x_1, \ldots, x_k\}\subset \Omega_{\delta} \atop x_i\neq x_j \mbox{\scriptsize\  for } i\neq j}
\sum_{I \subset \{x_1, \ldots, x_k\} \atop J=\{x_1, \ldots, x_k\}\backslash I}\mathrm{E}_{\Omega_{\delta}}^{\mathrm{ref}}\Big[\prod_{u \in I}\eta_u\prod_{v \in J}\lambda\sigma_v\omega_v\Big]
\end{align*}
satisfies $\lim_{M\to\infty} \limsup_{\delta\downarrow 0}\E[R_{M, \delta}^2]=0$.

By the same calculations in \eqref{4.10} and \eqref{2.8}, we can apply Lemma \ref{L:ortho} and the inequality $(\sum_{i=1}^{n}a_i)^2\leq n\sum_{i=1}^{n}a_i^2$ to obtain
\begin{align}\label{4.40}
\mathbb{E}[R_{M, \delta}^2]\leq\sum_{k=M+1}^\infty
\sum_{\{x_1, \ldots, x_k\}\subset \Omega_{\delta} \atop x_i\neq x_j \mbox{\scriptsize\  for } i\neq j}
\!\!\!\! 2^k \!\!\! \sum_{I \subset \{x_1, \ldots, x_k\} \atop J=\{x_1, \ldots, x_k\}\backslash I} \!\!\!\!
\mathrm{E}_{\Omega_{\delta}}^{\mathrm{ref},\otimes 2}\left[\prod_{u \in I}\mathbb{E}[\eta_{u}(\sigma, \omega)\eta_{u}(\sigma',\omega)]\prod_{v \in J}\lambda^2\sigma_v\sigma_v'\right],
\end{align}
where $\sigma'$ is the independent copy of $\sigma$. We decompose the r.h.s\ of \eqref{4.40} into two parts corresponding respectively to $I=\emptyset$, which gives the dominant contribution, and $I\neq \emptyset$:
\begin{align}
S_{M, \delta}^{(0)} & :=\sum_{k=M+1}^\infty 2^k \lambda^{2k}
\!\!\!\!\!\! \sum_{\{x_1,\ldots, x_k\} \subset \Omega_\delta \atop x_i\neq x_j \mbox{\scriptsize\  for } i\neq j} \!\!\!\! \mathrm{E}_{\Omega_{\delta}}^{\mathrm{ref}}[\sigma_{x_1} \ldots \sigma_{x_k}]^2, \notag \\
S_{M, \delta}^{(1)} & := \sum_{k=M+1}^\infty 2^k\sum_{\{x_1,\ldots, x_k\} \subset \Omega_\delta \atop x_i\neq x_j \mbox{\scriptsize\  for } i\neq j } \sum_{I \subset \{x_1, \ldots, x_k\} \atop |I|\geq 1} \!\!\!\!
\mathrm{E}_{\Omega_{\delta}}^{\mathrm{ref},\otimes 2}\left[\prod_{u \in I}\mathbb{E}[\eta_{u}(\sigma, \omega)\eta_{u}(\sigma',\omega)]\prod_{v \in J}\lambda^2\sigma_v\sigma_v'\right]. \label{4.41}
\end{align}
It follows directly from assumption {\bf (A2)} that $\lim_{M\to\infty} \limsup_{\delta\downarrow 0}S_{M, \delta}^{(0)} =0$.

To prove \eqref{eq:Lim2}, it only remains to show that
\begin{equation}\label{2.36}
\lim_{M\to\infty} \limsup_{\delta\downarrow 0}S_{M, \delta}^{(1)} =0.
\end{equation}
In fact, we will show that $\limsup_{\delta\downarrow 0}S_{M, \delta}^{(1)} =0$. Our strategy is similar to the proof of \eqref{4.11}, except now we need to control the sum over $k\geq M+1$.

Recall from \eqref{4.21} the expansion
$$
\mathbb{E}[\eta_{u}(\sigma, \omega)\eta_{u}(\sigma',\omega)]
=\sum_{m=4}^{\infty}\sum_{l=2}^{m-2}\lambda^{m}a_{m,l}
(\sigma_u)^{m-l}(\sigma'_u)^{l}.
$$
Substituting this into \eqref{4.41} gives
$$
S_{M, \delta}^{(1)}=
\sum_{k=M+1}^\infty 2^k\sum_{\{x_1,\ldots, x_k\} \subset \Omega_\delta \atop x_i\neq x_j \mbox{\scriptsize\  for } i\neq j } \sum_{I \subset \{x_1, \ldots, x_k\} \atop |I|\geq 1} \!\!\!\!
\mathrm{E}_{\Omega_{\delta}}^{\mathrm{ref},\otimes 2}\left[\prod_{u \in I}\Big(\sum_{m=4}^{\infty}\sum_{l=2}^{m-2}\lambda^{m}a_{m,l}
(\sigma_u)^{m-l}(\sigma'_u)^{l}\Big) \prod_{v \in J}\lambda^2\sigma_v\sigma_v'\right].
$$
Let us first consider the contributions from a fixed $k\geq M+1$. Denote $I=\{u_1,\cdots, u_{|I|}\}\neq\emptyset$ and $J=\{x_1, \ldots, x_k\}\backslash I$. By the same calculations as those leading to \eqref{2.34}, we have the bound
\begin{align}
&\sum_{\{x_1,\ldots, x_k\} \subset \Omega_\delta \atop x_i\neq x_j \mbox{\scriptsize\  for } i\neq j } \sum_{I \subset \{x_1, \ldots, x_k\} \atop |I|\geq 1}  \left|\mathrm{E}_{\Omega_{\delta}}^{\mathrm{ref},\otimes 2}\left[\prod_{u \in I}\left(\sum_{m=4}^{\infty}\sum_{l=2}^{m-2}\lambda^{m}a_{m,l}
(\sigma_u)^{m-l}(\sigma'_u)^{l}\right)\prod_{v \in J}\lambda^2\sigma_v\sigma_v'\right]\right|\nonumber\\
\leq\ & \sum_{\substack{\iota_1, \iota_2, \iota_3\geq 0 \\ \iota_1+\iota_2\geq 1 \\ \iota_1+\iota_2+\iota_3=k }} \sum_{\substack{I_1, I_2, J \subset \Omega_\delta \\ I_1\cap I_2=\emptyset, I \cap J=\emptyset \\ |I_1|=\iota_1, |I_2|=\iota_2, |J|=\iota_3}}
\!\!\!\!\!\!\!\!\!\!\!\!
\lambda^{2|J|} \Bigg(\prod_{i=1}^{|I|}(\lambda C)^{m_i}|a_{m_i,l_i}|\Bigg) \e_{\Omega_{\delta}}^{\mathrm{ref}}\left[\prod_{u \in I_1}\sigma_u\prod_{v \in J}\sigma_v\right]^2 \notag \\
\leq\ & \sum_{\substack{\iota_1, \iota_2, \iota_3\geq 0 \\ \iota_1+\iota_2\geq 1 \\ \iota_1+\iota_2+\iota_3=k }} \sum_{\substack{I_1, I_2, J \subset \Omega_\delta \\ I_1\cap I_2=\emptyset, I \cap J=\emptyset \\ |I_1|=\iota_1, |I_2|=\iota_2, |J|=\iota_3}}
\!\!\!\!\!\!\!\!\!\!\!\!   \lambda^{2|J|} (\hat{C}\lambda)^{4|I|} \e_{\Omega_{\delta}}^{\mathrm{ref}}\left[\prod_{u \in I_1}\sigma_u\prod_{v \in J}\sigma_v\right]^2 \notag \\
\leq\ &
\sum_{\substack{\iota_1, \iota_2, \iota_3\geq 0 \\ \iota_1+\iota_2\geq 1 \\ \iota_1+\iota_2+\iota_3=k }}
\tilde C^{\iota_1+\iota_2+\iota_3} \delta^{d(\iota_1+\iota_2) -2\gamma(\iota_1 + 2\iota_2)} \frac{\Vert \psi_\delta\Vert^2_{L^2(\Omega^{\iota_1+\iota_3})}}{(\iota_1+\iota_3)!}.
\label{4.44a}
\end{align}
Summing this bound over $k\geq M+1$ then gives
\begin{align*}
|S_{M, \delta}^{(1)}| & \leq \sum_{k=M+1}^\infty 2^k \sum_{\substack{\iota_1, \iota_2, \iota_3\geq 0 \\ \iota_1+\iota_2\geq 1 \\ \iota_1+\iota_2+\iota_3=k }}
\tilde C^{\iota_1+\iota_2+\iota_3} \delta^{d(\iota_1+\iota_2) -2\gamma(\iota_1 + 2\iota_2)} \frac{\Vert \psi_\delta\Vert^2_{L^2(\Omega^{\iota_1+\iota_3})}}{(\iota_1+\iota_3)!} \notag \\
& \leq \sum_{\iota_1, \iota_3\geq 0}  \frac{\tilde C^{\iota_1+\iota_3}}{(\iota_1+\iota_3)!} \Vert \psi_\delta\Vert^2_{L^2(\Omega^{\iota_1+\iota_3})} \sum_{\iota_2\geq M+1-\iota_1-\iota_3\atop \iota_2\geq 0,\, \iota_1+\iota_2\geq 1} (\tilde C \delta^{d-4\gamma})^{\iota_2} \delta^{(d-2\gamma)\iota_1} \notag \\
& \leq (\delta^{d-2\gamma}+ 2 \tilde C \delta^{d-4\gamma}) \sum_{m=0}^\infty  \frac{(2\tilde C)^m}{m!} \Vert \psi_\delta\Vert^2_{L^2(\Omega^m)},
\end{align*}
where the two terms in the prefactor on the last line come from $\iota_2=0$ and $\iota_2\geq 1$ respectively, provided $\delta>0$ is sufficiently
small such that $\tilde C \delta^{d-4\gamma}<1/2$. Note that the sum on the last line is uniformly bounded as $\delta\downarrow 0$ by assumptions {\bf (A1)} and {\bf (A2)}. Since we assume $\gamma<d/4$, \eqref{2.36} follows
immediately. This concludes the proof of \eqref{eq:Lim2}.
\qed

\bigskip

\noindent
{\bf Acknowledgements.} We thank Federico Camia for helpful discussions on the assumption {\bf (A3)}, and we also thank the referee for helpful comments on the paper. L.~Li acknowledges the support of the China Scholarship Council program (Project ID: 202206370088). V.~Margarint would like to thank the National University of Singapore for the hospitality. R.~Sun is supported by NUS grant A-8001448-00-00.

\bibliography{MyBib}
\bibliographystyle{abbrv}

\end{document}